\pdfoutput=1
\documentclass[11pt, reqno]{amsart}
\usepackage{amssymb,latexsym,amsmath,amsfonts}
\usepackage{latexsym}
\usepackage[mathscr]{eucal}
\usepackage[english]{babel}
\usepackage{graphicx}
\usepackage{epsfig}
\usepackage{color}

\numberwithin{equation}{section}

\theoremstyle{definition}
\newtheorem{definition}{Definition}[section]

\theoremstyle{remark}
\newtheorem{remark}[definition]{Remark}
\theoremstyle{plain}
\newtheorem{theorem}[definition]{Theorem}
\newtheorem{lemma}[definition]{Lemma}
\newtheorem{proposition}[definition]{Proposition}

\newtheorem{corollary}[definition]{Corollary}

\begin{document}

\title{Relative Symplectic Caps, 4-Genus and Fibered Knots}

\author{Siddhartha Gadgil and Dheeraj Kulkarni}
\address{Department of Mathematics, Indian Institute of Science, Bangalore, India 560012.\\
School of Mathematics, Georgia Institute Of Technology, Atlanta, US. }
\email{gadgil@math.iisc.ernet.in, kulkarni@math.gatech.edu}

\thanks{}

\keywords{Relative Symplectic Caps, 4-genus and fibered knots}

\begin{abstract}
We prove relative versions of the symplectic capping theorem and sufficiency of Giroux's criterion for Stein fillability and use these to study the $4$-genus of knots. 

More precisely, suppose we have a symplectic 4-manifold $X$ with convex boundary and a symplectic surface $\Sigma$ in $X$ such that $\partial \Sigma $ is a transverse knot in $\partial X $. In this paper, we prove that there is a closed symplectic 4-manifold $Y$ with a closed symplectic surface $S$ such that $(X, \Sigma) $ embeds into $(Y, S) $ symplectically. As a consequence we obtain a relative version of the Symplectic Thom conjecture. 

We also prove a relative version of the sufficiency part of Giroux's criterion for Stein fillability, namely, we show that a fibered knot whose mondoromy is a product of positive Dehn twists bounds a symplectic surface in a Stein filling. We use this to study 4-genus of fibered knots in $\mathbb{S}^3 $. Further, we give a criterion for quasipostive fibered knots to be strongly quasipositive.  
\end{abstract} 
\maketitle

\section{Introduction}\label{S:intro}

By a theorem of Eliashberg~\cite{E3} and Etnyre \cite{Et3}, any compact symplectic $4$-manifold with convex boundary embeds in a closed symplectic $4$-manifold. This has many applications, in particular it played an important role in Kronheimer-Mrowka's proof of property~P (in fact of the stronger result that the fundamental group of the $3$-manifold obtained by $\pm1$-surgery about a non-trivial knot in $S^3$ has a non-abelian $SU(2)$ representation). Our first main result is a relative version of this result for a compact symplectic $4$-manifold containing a symplectic surface. 

\begin{theorem}[Relative Symplectic Capping]\label{cap} Let $(X, \omega)$ be a symplectic 4-manifold with convex boundary. Let $\Sigma $ be a symplectic surface in $X$ such that $\partial \Sigma$ is a transverse link contained in $\partial X $. Then there exists a closed symplectic 4-manifold $(Y, \omega') $ and a closed symplectic surface $S\subset Y$ such that $(X, \Sigma) $ embeds symplectically into $(Y, S) $. 
\end{theorem}

This is motivated by the so called \emph{Symplectic Thom Conjecture}, which is a theorem~\cite{OS} saying that symplectic surfaces minimise genus in their respective homology classes. As an immediate corollary to the above result, using the Symplectic Thom Conjecture we obtain the following result. This has been proved previously using indirect arguments.

\begin{theorem}[Relative Version of Symplectic Thom Conjecture]\label{relThom} Let $(X, \omega) $ and $\Sigma $ be as above. Then $\Sigma $ is genus-minimizing in its relative homology class.
\end{theorem}
A relative version of Symplectic Thom conjecture is also proved by Bowden (see section 7.3 in \cite{JB}) under the assumption that $H_2(X)=0 $. Bowden uses Gay's symplectic 2-handles (see \cite{Gay}) to cap off the symplectic surface by adding \textit{only} discs. In this paper, we prove the relative version of symplectic Thom Conjecture without any additional hypothesis on $X$, in particular without the hypothesis that $H_2 (X)=0 $.

A relative version of the Symplectic Thom Conjecture can be established using Mrowka-Rollin's generalization (\cite{MR}) of Bennequinn inequality for Legendrian knots. 
The proof of relative version of Symplectic Thom conjecture in this article is based on the most intuitive and simple geometric idea that we can construct relative symplectic caps to cap off a symplectic subsurface $S$ first and then the ambient manifold $X$.

We shall apply Theorem~\ref{relThom} in particular to study the $4$-genus of a link in $S^3$ by proving a relative version of Giroux's sufficient condition for bounding a Stein manifold (see Theorem \ref{RelGiroux} below). First observe that if a link $K$ in $\partial\mathbb{B}^4 $ bounds a  symplectic surface $\Sigma $ in $\mathbb{B}^4 $, then, by Theorem~\ref{relThom},  $\Sigma $ realizes the $4$-genus of the link.
 
\begin{corollary} \label{S3Cap}
Let $L$ be transverse link in $(\mathbb{S}^3, \xi_{st}) $. Let $\Sigma$ denote a symplectic surface in $(\mathbb{B}^4, \omega_{st}) $ such that $\Sigma \cap \mathbb{S}^3= \partial \Sigma= L $. Then there exists a closed symplectic 4-manifold $(X, \omega)$ with a closed symplectic surface $S$ such that $(\mathbb{B}^4, \Sigma)$ can be embedded symplectically into $(X, S)$.  
\end{corollary}

A link $L$ in $\mathbb{S}^3 $ is called \emph{fibered} if $\mathbb{S}^3 -L$ fibres over $\mathbb{S}^1 $. Fibered links give rise to \emph{open book decompositions} of $\mathbb{S}^3$. A fundamental result of Giroux gives a one to one correspondence between contact structures (up to isotopy) and open book decompositions up to positive stabilizations. Thus, we can study fibered links using contact topology. 

There is a dichotomy among contact structures: a contact structure is either \emph{tight} or \emph{overtwisted}. Therefore, we can classify fibered links into two classes, namely tight and overtwisted.
A fibered link $L$ is called tight (overtwisted) if corresponding open book for $ \mathbb{S}^3 $ supports tight (overtwisted) contact structure on $\mathbb{S}^3$. Note that there is a unique, up to isotopy, tight contact structure on $\mathbb{S}^3 $. A characterization of tight links appears in \cite{H1} (See proposition 2.1 in \cite{H1}). Namely, a fibered link is tight if and only if it is strongly quasipositive. 

If a link in $\mathbb{S}^3 $ bounds a complex curve in $\mathbb{B}^4 $ then there are strong topological implications on the link. Links which arise as transverse intersection of a complex curve in $\mathbb{C}^2 $  with $\mathbb{S}^3 $ are called as transverse $\mathbb{C} $-links. Rudolf showed that quasipositive links are transverse $\mathbb{C} $-link (see \cite{R2}). The converse was shown by Boileau and Orevkov (see \cite{B1}). Plamenevskaya showed that for quasipositive knots concordance invariant and smooth 4-genus are equal (see \cite{P1}).

In a similar spirit, we may ask whether a link in $\mathbb{S}^3 $ is isotopic to the boundary of a symplectic surface in $\mathbb{B}^4 $ endowed with standard symplectic structure $ \omega_{st}$. From the above discussion we know that quasipositive links in $\mathbb{S}^3 $ do bound symplectic surfaces in $\mathbb{B}^4 $. We shall show that fibered links with monodromy a product of positive Dehn twists bound symplectic surfaces, and these have genus equal to genus the fiber of the fibration. Before stating this result, we fix some notation, which we use throughout this paper.
 
Let $L$ be a fibered link in $\mathbb{S}^3$. Let $(\mathbb{S}^3,K)$ be the open book decomposition given by it. Let $\Sigma$ denote the page of the open book and $\phi$ denote the monodromy of the open book. For a curve $\gamma$ on $\Sigma$, we denote the positive (right handed) Dehn twist along $\gamma$ by $D_{\gamma}$. 
 Let $\xi_{st}$ denote the unique (up to isotopy) tight contact structure on $\mathbb{S}^3 $ and $ \omega_{st}$ denote the standard symplectic form on $ \mathbb{B}^4  \subset \mathbb{R}^4$. Let $g(K)$ denote the Seifert genus and $g_4(K)$ denote the 4-genus of $K$. 

The following result is a refinement of a result of Giroux giving a sufficient condition for bounding a Stein Manifold.

\begin{theorem}\label{RelGiroux}
Let $L$ be a fibered link in $\mathbb{S}^3$. If the corresponding open book $( \Sigma,\phi) $ has monodromy satisfying the following  
$$ \phi = D_{\gamma_1}\circ D_{\gamma_2} \circ ... \circ D_{\gamma_k} $$ where $\gamma_i$'s are nonseparating curves on $\Sigma$ then $L$ bounds a symplectic surface in $(\mathbb{B}^4,\omega_{st})$ which is homeomorphic to $\Sigma$. 
\end{theorem}

As a consequence, we obtain the following corollaries.

\begin{corollary} \label{4G3G}
Let $L$ and $(\Sigma, \phi) $ be  as given in Theorem~\ref{RelGiroux}. Then  $g_4(L)= g(\Sigma) $, in other words, 4-genus of $L$ equals the genus of the page of open book $(\Sigma, \phi)$. Therefore, $g_4(L)=g(L) $.
\end{corollary} 

This also gives a sufficient condition for a fibered link to be strongly quasipositive. 

\begin{corollary}[A Criterion for Strong Quasipositivity]\label{sqp}
If $K$ is a fibered knot such that corresponding open book 
has monodromy given by product of positive Dehn twists then $K$ is strongly quasipositive.
\end{corollary}
This follows from Hedden's characterization of fibered knots that are strongly quasipositive as the fibered knots for which 3-genus and 4-genus are equal (see Corollary 1.6 in \cite{H1}). 

{\bf Acknowledgements:}
 The second author would like to thank Peter Kronheimer for communicating through MathOverflow an outline of a proof of a relative version of symplectic Thom conjecture using Mrowka-Rollin's result. He would also like to thank John Etnyre for going through a draft of this article and making useful suggestions to bring clarity in exposition.

\section{Preliminaries}

In this section, we briefly recall some definitions and results which are used in the proofs presented here.

\begin{definition} Let $M$ be a closed oriented 3-manifold. A contact structure $\xi $ is a plane distribution such that if $\alpha$ is a locally defined 1-form with $ker(\alpha)= \xi$ then $\alpha \wedge d \alpha \neq 0$.
\end{definition}

\begin{definition}
An \textit{open book decomposition} of a 3-manifold $M$ is a pair $(B, \pi) $ where \\
(a). $B $ is a link in $M$, called the binding of the open book. \\
(b). $\pi : M \backslash B \rightarrow \mathbb{S}^1$ is a fibration. For each $\theta \in \mathbb{S}^1 $, $\pi^{-1}(\theta)$ is an interior of a compact surface $\Sigma_{\theta} $ with $\partial \Sigma_{\theta}= B$. \\
 We call  $\Sigma_{\theta} $ a page of the book.
 Let $\phi $ denote the monodromy of the fibration $\pi : M \backslash B \rightarrow \mathbb{S}^1 $. Then we get an abstract open book $(\Sigma, \phi)$.
 \end{definition}

\begin{definition}
A contact structure $\xi $ on $M$ is supported by an open book decomposition $(B, \pi) $ of $M$ if $\xi$ can be isotoped through contact structures so that there is a contact 1-form $ \alpha$ for $\xi $ such that
\begin{enumerate}
\item $d \alpha $ is a positive area form on each page $\Sigma_{\theta} $ the open book.
\item $\alpha >0 $ on $B$. 
\end{enumerate} 

\end{definition}

\begin{theorem}[Giroux, 2002, \cite{Gi}] 
Let $M$ be a closed, oriented $3$-manifold.
\begin{enumerate}
\item Given an open book decomposition of $M$, there is a positive contact structure, unique up to isotopy, supported by the open book.
\item Every contact structure on $M$ is supported by an open book.
\item Two open book structures on $M$ that support isotopic contact structures are related by positive stabilizations.
\end{enumerate}

\end{theorem}

Giroux's theorem allows us to deal with open books rather than contact structures to construct symplectic fillings. We next recall a theorem about symplectic handle addition.

\begin{theorem}[Eliashberg, 1990, Weinstein 1991, \cite{E1,W1}] If $(X, \omega)$ is a symplectic 4-manifold with strongly (weakly) convex boundary and $X'$ is obtained from $X$ by attaching a 2-handle to $X$ along a Legendrian knot in $\partial X $ with framing one less than contact framing (also called as Thurston-Bennequin framing), then $\omega $ extends to a symplectic form $\omega' $ on $X'$ in such a way that $X'$ has strongly (weakly) convex boundary.
\end{theorem}

We call a symplectic handle addition as in the above theorem a Legendrian surgery.
Next we state a version of the Legendrian Realization Principle (LeRP) which is useful in realizing curves on the page of open book as Legendrian curves to carry out Legendrian surgery along them.

\begin{proposition} If $\gamma $ is a nonseparating curve on a page of an open book which supports a contact structure $\xi $ then we can isotop the open book slightly so that $\gamma $ is Legendrian and the contact framing agrees with the page framing.
\end{proposition}

LeRP can be found in  generality in a paper of Ko Honda (see theorem 3.7 in \cite{KH}). The above version is an easy consequence of LeRP and appears in a paper of Etnyre (see 6.8 in \cite{Et3}). 
We need three more ingredients.

\begin{definition} Recall that $X$ is a Stein filling of contact manifold $(M. \xi) $ if $X$ is a sublevel set of a plurisubharmonic function on a Stein surface and the contact structure induced on $\partial X $ is contactomorphic to $(M, \xi) $.
\end{definition}

\begin{theorem}[Eliashberg, \cite{E2}] There is a unique Stein filling of $(\mathbb{S}^3, \xi_{st})$.
\end{theorem}

\begin{theorem}
[Eliashberg~\cite{E3}, Etnyre~\cite{Et1}]If $(X, \omega) $ is a compact symplectic manifold with weakly convex boundary then there is a closed symplectic $(X', \omega') $ into which $(X, \omega) $ symplectically embeds.
\end{theorem}

\begin{theorem} 
[Ozsv\'{a}th-Szab\'{o}, 2000, \cite{OS} ]\label{STC} An embedded symplectic surface in a closed, symplectic four-manifold is genus-minimizing in its homology class.
\end{theorem}

\section{Construction of Relative Symplectic Caps.}

In this section, we construct relative symplectic caps  using which we give a proof of Theorem \ref{cap}.

\begin{definition} Let $X$ be a symplectic 4-manifold with contact type boundary. Let $\Sigma $ be a symplectic surface in $X$ with $\partial \Sigma \subset \partial X$. We say that the \textit{pair $(X, \Sigma)$ can be capped symplectically} if there exists a closed symplectic manifold $Y$ with a closed symplectic surface $S$ such that the pair $(X, \Sigma)$ embeds into $(Y,S) $ symplectically. In this case, we say that $(X, \Sigma) $ \textit{is capped to a pair $(Y, S) $ symplectically}.
 \end{definition}

To prove Theorem 1.1, we need to show that $(X , \Sigma) $ can be capped symplectically. We first prove a useful proposition for gluing open sets from symplectic manifold to get an open symplectic manifold. Our proof of Theorem 1.1 is a refinement of the argument presented in Theorem 4 of \cite{Et2}. While broadly following the same line of argument we need to take care of following two issues.

\begin{enumerate}
\item Firstly, the resulting manifold should have weakly convex boundary so that we can use the symplectic capping theorem of Eliashberg and Etnyre at a later stage.
\item Secondly and most importantly, we must make sure that the symplectic surface $\Sigma $ is extended to a closed symplectic surface in the new symplectic manifold. 
\end{enumerate}

The following Lemma gives the relative version of gluing along a diffeomorphism that respects \emph{contact forms} (not just contact structures).

\begin{lemma}\label{glueeasy} Let $U_i$ be a domain in symplectic manifold $(X_i, \omega_i)$ with the boundary of $U_i$ piecewise smooth, for $i=1,2 $. Let $S_i $ be a smooth piece in $\partial \overline{U}_i $. Let $T_i $ be a smooth submanifold in $X_i $ which contains $S_i$ in its interior.
Assume the following:  
\begin{enumerate}
\item There is a symplectic dilation $v_1$ transverse to $ S_1$ pointing out of $U_1 $. Let $\alpha_1 = i_{v_1} \omega_1$.
\item There is a symplectic dilation $v_2$ transverse to $S_2$ pointing into $U_2 $. Let $\alpha_2 = i_{v_2} \omega_2$.
\item There is a smooth map $\phi : T_1 \rightarrow T_2 $ such that $\phi(S_1)= S_2 $ and $ \phi^{\ast} \alpha_2 = \alpha_1 $.
\end{enumerate} 
Then $ Y_0 = U_1 \cup_{S_1} U_2 $ admits a symplectic structure.
\end{lemma} 
   
\begin{proof} By hypothesis, $\alpha_i = i_{v_i}\omega_i$ are the contact forms induced on $S_i $. Consider the symplectization $Symp(S_i, \xi_i)= (0, \infty) \times S_i$, where $\xi_i= ker ( \alpha_i) $, for $i =1,2 $. Then we have a neighborhood $N_1 $ of $ S_1$ in $X_1$ symplectomorphic to a neighborhood $N'_1 $ of $\alpha_1 (S_1) = \{ 1\} \times S_1 $ in $Symp (S_1, \xi_1) $ such that symplectic dilation $v_1$ is taken to a vector field pointing in the $\frac{\partial}{\partial t} $ direction which we denote by $v_1 $ itself. Similarly there is a neighborhood $N_2 $ of $ S_2$ symplectomorphic to a neighborhood $N'_2 $ of $\alpha_2(S_2)= \{ 1\} \times S_2 $ (see Figure~\ref{dilations}).

\begin{figure}[h!]
\centering
\includegraphics[height=80mm, width=100mm]{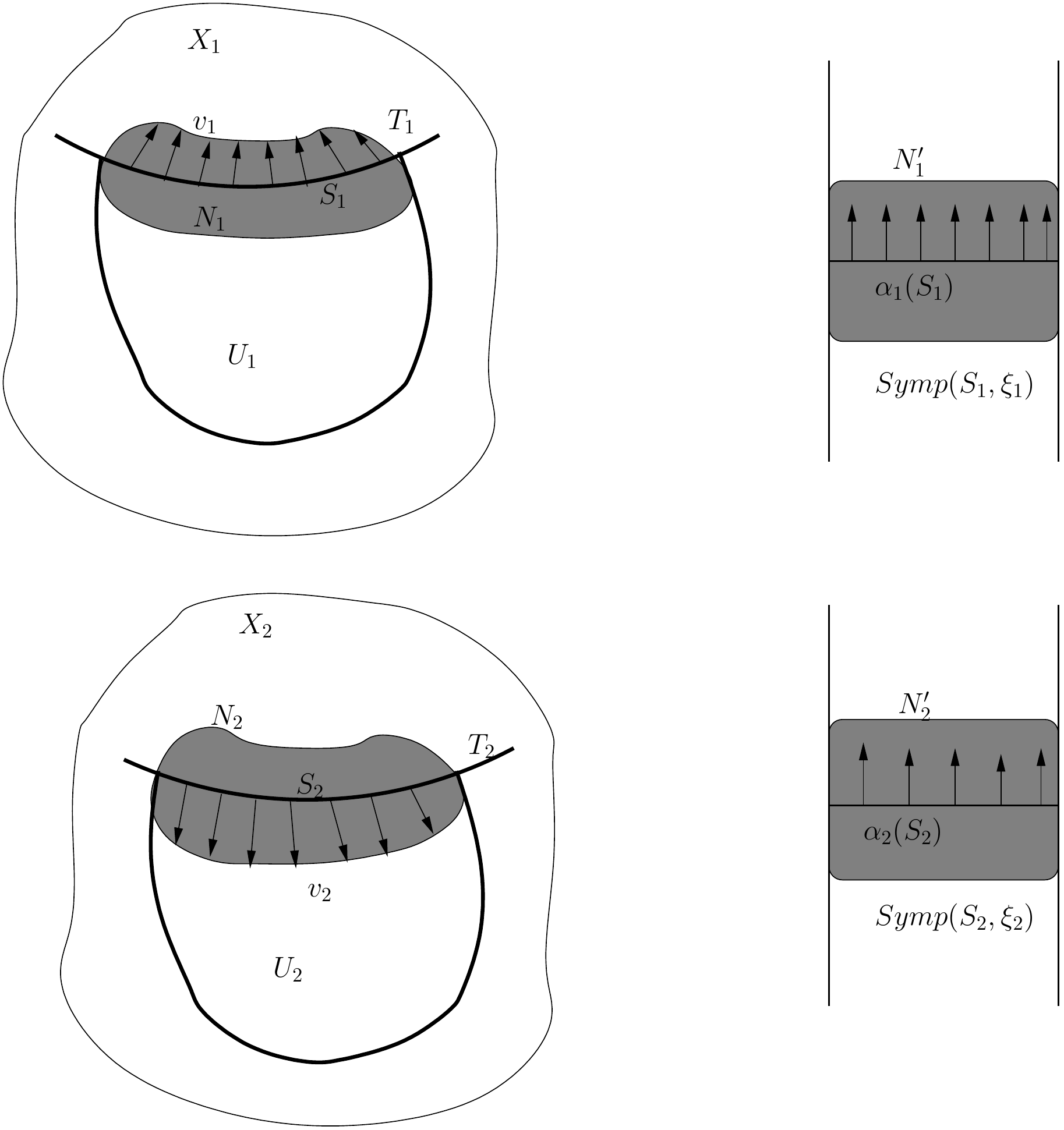}
\caption{Symplectizations And Dilations}\label{dilations}
\end{figure}


By hypothesis, $ \phi$ is a contactomorphism between $(S_1, ker(\alpha_1)) $ and $(S_2, ker(\alpha_2)) $ such that $\phi^{\ast} \alpha_2 = \alpha_1 $. We can also view $\phi $ as a contactomorphism between $\alpha_1(S_1) $ and $\alpha_2(S_2) $. Now we extend $\phi $ to a symplectomorphism $\phi'= id \times \phi $. An easy check shows that $\phi' $ takes symplectic dilation $v_1 $ to $v_2 $. We regard $\phi' $ as a symplectomorphism between $N_1 $ and $N_2$, by suitably shrinking $N_1$ and $N_2 $ if necessary. 
 Let $U^0_i = U_i \cup N_i $.  We may now use the symplectomorphism constructed above to identify $N_1 \subset U^0_1 $ with $N_2 \subset U^0_2$. Thus we get $Y_0 = U^0_1 \cup_{N_1} U^0_2  $ with a symplectic form $\omega $ such that $\omega \vert_{U^0_i}= \omega_i $, for $i=1,2$. 
\end{proof}

We now turn to the proof of Theorem \ref{cap}. We, first, construct a symplectic cap and then glue it using Lemma~\ref{glueeasy}. 
\begin{proof}[Proof of Theorem 1.1] 
We, first, outline the construction of relative symplectic cap. We embed surface $\Sigma$ in a closed surface $S$. A disc bundle $E$ over $S$ with positive Chern class is a symplectic manifold with (the $0$-section) $S$ a symplectic surface in it. The cap is isotopic to the complement of $E\vert_{\Sigma}$, but with the boundary of the deleted region chosen to ensure convexity and infinite order of contact with the boundary of $E$.

For simplicity, we assume that link $L$ has only one component, i.e., $L$ is a knot in which case we prefer to denote it by $K$. By inductively applying the following construction of the cap for each component of $L$, we can obtain a suitable relative symplectic cap when the link $L$ has more than one component.

\noindent \textbf{Construction of a Relative Symplectic Cap:} \\
We now turn to more details. Let $S$ be a closed Riemann surface into which $\Sigma $ embeds symplectically. Let $J_S $ denote the almost complex structure on $S$. Let $\omega_S $ denote a symplectic form on $S$ compatible with $J_S $. Let $E$ be a complex line bundle over $S$ with $c_1(E)= - k[\omega_S] $ for a fixed $k\in \mathbb{N} $. Let $J= J_S \oplus J_{st} $, where $J_{st}$ denotes the standard almost complex structure on the fiber given by multiplication by $i$.
  
Let $\omega $ denote the symplectic form on the total space given by $\omega|_{(x,v)}= \pi^{\ast}(\omega_S) + \omega_x|_v $, where $\omega_x $ denotes the symplectic form on the fiber at $x \in S$ which is compatible with $J_{st} $. Observe that the almost complex structure $J$ is compatible with $\omega$.
Thus, we have a Hermitian line bundle $(E, \omega,J, g )$. Let $\Sigma ' \subset S$ be a surface which contains $\Sigma $ in its interior and $\partial \Sigma ' $ is parallel to $\partial \Sigma $. As $\Sigma'$ has non-empty boundary, we see that $E|_{\Sigma'} $ is isomorphic to the trivial Hermitian line bundle $(\Sigma' \times \mathbb{C},\omega_{st}, J_{st}, g_{st} ) $. Denote this isomorphism by $\Psi$ from $\Sigma' \times \mathbb{C} $ to $E|_{\Sigma'} $. By the symplectic neighborhood theorem,
there are symplectomorphic neighborhoods $N_1$ of  $\Sigma' \times \{ 0\} $ and $N_2$ of $\Psi (\Sigma'\times \{ 0\})$. Let $\chi : N_1 \rightarrow N_2 $ denote the above symplectomorphism. 

Since $c_1(E)= -k[\omega_S] $ with $k \in \mathbb{N} $, a sufficiently small disc bundle in $E$ is symplectically convex (see Proposition 5 in \cite{Et2}). We denote this disc bundle by $\hat{E}$. We can choose a disc bundle $\hat{E}$ such that $\hat{E}|_{\Sigma'} \subset N_2$. 
 We have a symplectic form $ \omega ' $ on the disc bundle so that the radial vector field along the fibre is a symplectic dilation, (see the proof of Proposition 5 in \cite{Et2}). We notice that the symplectomorphism in a proof
 of symplectic neighborhood theorem (see the proof of Theorem 3.30 in \cite{MS}) is chosen in such a way that it takes the radial vector field to a vector field tangent to the fiber of the disc bundle. We pull back this vector field using $\chi$ onto $N_1 \subset \Sigma' \times \mathbb{C}$.
Therefore, we have a symplectic dilation on $N_1$ tangent to fibers.   

Now, we turn our attention to a sufficiently small trivial disc bundle in $N_1$ endowed with a symplectic dilation transverse to the boundary pointing out of disc bundle.  Let $E'$ denote this disc bundle.  
 
\begin{lemma}\label{pshsurf} Given $\epsilon >0 $, there exists a smooth non-constant subharmonic function $\phi : \Sigma \rightarrow \mathbb{R} $ such that $\phi|_{\partial \Sigma} =1 $ and $|\phi| < \epsilon $ away from a neighborhood of $\partial \Sigma $. 
\end{lemma}
\begin{proof}
We construct the function $\phi$ by solving the Dirichlet problem $\Delta \phi=f$ on $\Sigma$ with boundary value $1$, so that $f$ is a positive function chosen so that $f$ is small away from a neighbourhood of $\partial\Sigma$.
\end{proof}

 By Lemma \ref{pshsurf}, for $\epsilon > 0$ there is a nonconstant smooth subharmonic function $\phi : \Sigma' \rightarrow \mathbb{R} $ with $\phi\vert_{\partial \Sigma'}= \epsilon$.  We have a plurisubharmonic function on $E'$ given by $\psi(z_1, z_2) = \phi(z_1) +| z_2|^2 $. The sublevel set $W'= \{z \vert \psi(z) \leq \epsilon \}$  is a compact Stein domain contained in $ E'$, hence it is symplectically convex. We will make a suitable choice of $\epsilon $ later.
 
  Let $A$ denote an annulus bounded by $\partial \Sigma $ and $\partial \Sigma' $. Let $\mathbb{D}_{r} := \{ z \in \mathbb{C} ; |z| \leq r \} $. Then $E'|_{A}$ is just the product $A \times \mathbb{D}_{t} $ for some $ t >0$. As discussed earlier, there is a symplectic dilation pointing out of $A \times \mathbb{D}_t $ near the boundary. 


 
 There exists a $\delta > 0 $ depending on $\epsilon $ such that there are tubular neighborhoods $S_1 \subset \partial X$ and $S_2 \subset \partial W' $ of the transverse knot $K = \partial \Sigma$ which are contactomorphic to $(\mathbb{S}^1 \times \mathbb{D}_{\delta}, Ker( d\theta + r^2 d \phi))$, where $\theta $ is a coordinate along $\mathbb{S}^1 $ and $(r, \phi) $ are polar coordinates on $\mathbb{D}_{\delta} $.   

\begin{figure}
\centering
\includegraphics{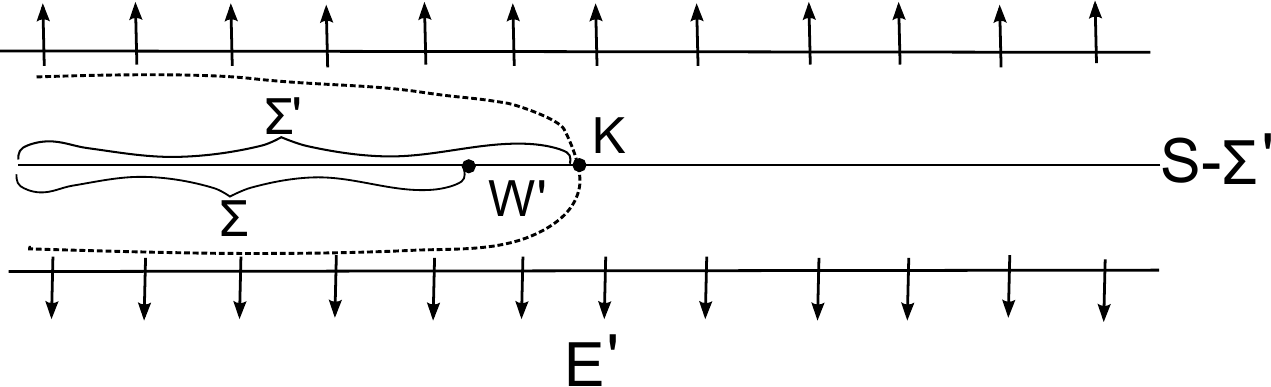}
\caption{The disc bundle with convex boundary.}
\end{figure}

The following Lemma ensures that a symplectic cap has a suitable shape. 
\begin{lemma}
There is a positively valued smooth function  $R: S \rightarrow \mathbb{R}$   such that the following conditions are satisfied:
 \begin{enumerate}
 \item The set $E''= \{v_x \in E' :
 \Vert v_x \Vert \leq R(x)\}$ is a disc bundle contained in $E'$. Notice that the radius of the disc bundle at point $x$ is given by $R(x)$. Further, the radial vector field is a symplectic dilation.
 \item The intersections of $E''$ and $ W' $ with $E'\vert_\Sigma$ coincide. 
 \item The boundary of $E''= \{v_x \in E' : \Vert v_x \Vert \leq R(x)\}$ intersects $S_2$ with infinite order of contact.
 \end{enumerate}
\end{lemma}   

\begin{remark}
 We would like to note that the above mentioned function $R$ depends on choice of $\epsilon$. Therefore, the set $ E''$ also depends on the choice of $\epsilon $.
\end{remark}

Observe that  $E''-W'$ is a symplectic manifold with the smooth piece $S_2 $ endowed with a symplectic dilation $v$ transverse to $S_2 $ pointing into $E''-W'$. We call $E''-W' $ a relative symplectic cap for $(X, \Sigma) $. See figure 3.


\begin{figure}
\centering
\includegraphics{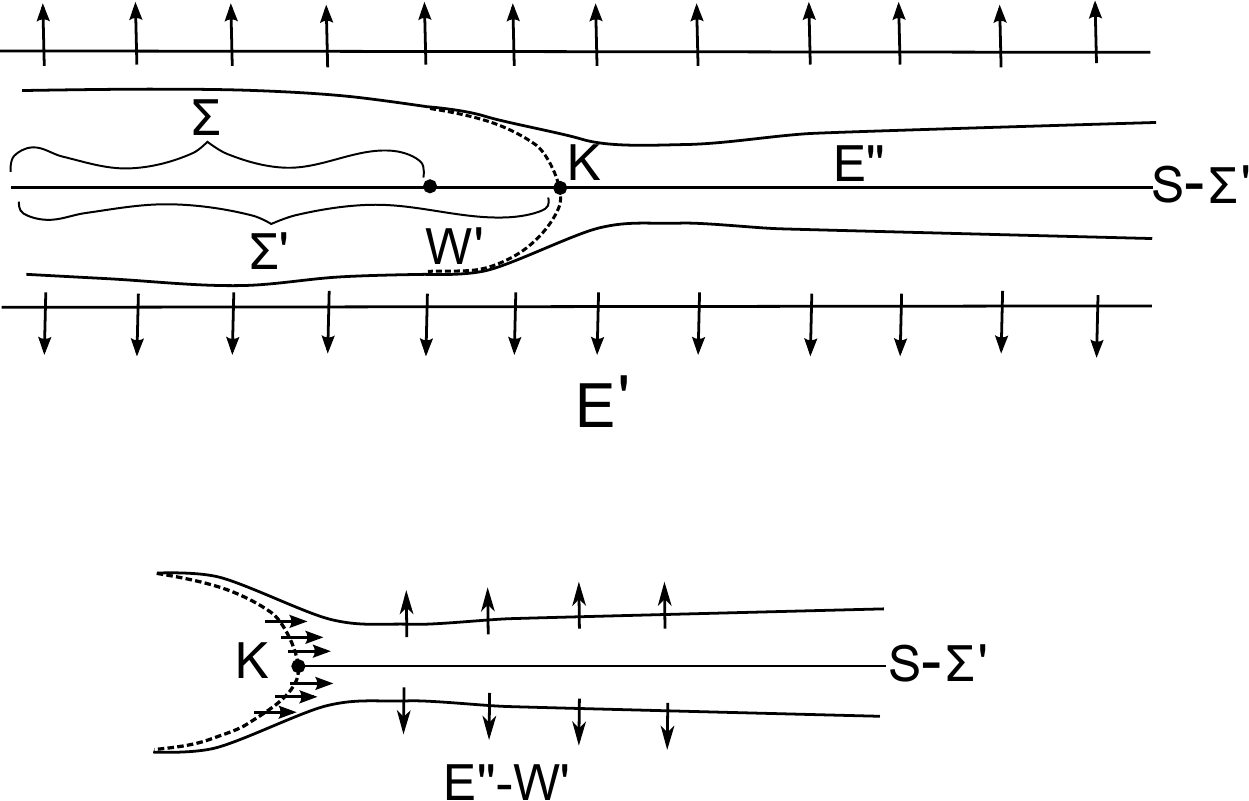}
\caption{The above figure shows a thematic picture of a relative symplectic cap. The dotted curve represents the neighbourhood $S_2$ of the knot $K$. The symplectic dilation on $S_2 $ points into $E'' -W' $. The symplectic dilation given by the radial vector field point out of $E''-W' $.}
\end{figure}

\noindent \textbf{Symplectic dilation on the relative cap:} \\
Now, we show that there is a symplectic dilation on $E''-W' $ which is 
transverse to its boundary and pointing into $E''-W' $ when restricted to $S_2$. 

Let $A$ denote an annulus bounded by $\partial \Sigma $ and $\partial \Sigma' $ such that by Lemma \ref{pshsurf} there exists a subharmonic function $\phi : \Sigma' \to \mathbb{R} $ and $ |\phi(z_1) |=|\psi- |z_2|^2| < \epsilon $ away from $A$. Therefore $|\nabla \psi - R | = O(\epsilon)$, where $R $ is the gradient vector field for the function $|z_2|^2 $ away from the set where $z_2 =0 $. 

Let $\alpha_R = \omega(R, .)$ and $ \alpha = \omega (\nabla \psi, .) $.  Then 
$\alpha_R - \alpha = \omega (R - (\nabla \psi) , .)$. This implies $|\alpha_R - \alpha |= O(\epsilon) $.

Let $\gamma $ be a simple closed curve on the annulus $A$ such that $[\gamma]$ generates $H_1(A, \mathbb{R})$. 
 We see that $ \int_{\gamma} ( \alpha_R - \alpha) = O(\epsilon) $. Let $ c= \int_{\gamma} ( \alpha_R - \alpha) $. Let $\alpha '=\alpha_R - \alpha - c \theta $, where $\theta $ is 1-form representing $PD([\gamma]) $. Then $\alpha' $ is closed and $[\alpha']=0 $, by construction, so $\alpha'$ is exact. Thus, $ \alpha' =df $ for some function $f$ on the annulus $A$. We see that $|df|= O(\epsilon) $ since $|\alpha_R - \alpha |=O(\epsilon) $ and $ c =O(\epsilon) $. Extend $f$ to a function $\tilde{f}$ on $A$ such that outside a neighborhood of $A$,  $\tilde{f} $ is zero.
 
Let $\beta = \alpha_R- c\theta  - d\tilde{f}$. Also notice that $ d \beta = \omega$. Let $X$ be the vector field satisfying $i_X \omega = \beta$. 

Now, take sufficiently small $\epsilon>0$, we see that $X$ is a symplectic dilation transverse to the boundary of $E''-W' $ and pointing into $E''-W' $ when restricted to $ S_2$. Observe that $\beta $ is a contact form and it agrees with $\alpha$ on the disc bundle restricted to $A$.

\begin{remark} Observe that by taking sufficiently small $\epsilon $, the symplectic dilations $X$ and the radial vector field $R$(which is also a symplectic dilation) are sufficiently close. Thus, it became possible to interpolate the two symplectic dilations without destroying transversality on the boundary of relative symplectic cap.  

\end{remark}

\noindent \textbf{Gluing of Relative Symplectic Cap to $(X, \Sigma) $:}\\
From above discussion, $S_1 $ and $S_2 $ are contactomorphic. Let $\phi: S_1 \rightarrow S_2 $ denote this contactomorphism. Then $\phi^{\ast} \alpha_2= \lambda \alpha_1 $, where $\lambda $ is a nonvanishing real valued function. By scaling the symplectic form on $E''-W' $ if necessary, we can assume that $\lambda >1$ on $S_1 $. We extend $\lambda $ to a smooth function $\lambda' $ on $\partial X $ such that away from a neighborhood $N$ of $S_1 $ it is identically equal to constant 1.
 We consider the graph of $\lambda' $ in $Symp(\partial X, ker(\alpha_1)) $ (see figure 3 above). Denote this graph by $\Gamma $. We observe that the graph $\Gamma $ has a contact form given by $\lambda \alpha_1 = \phi^{\ast} \alpha_2 $. Further, away from the neighborhood $N$ of $S_1$, the graph is given by $\alpha_1(\partial X) $. There is a neighborhood $N_1 $ of $S_1$ symplectomorphic to $N'_1 $ in $Symp(\partial X, ker(\alpha_1)) $.  Let $T$ denote the region bounded by $\alpha_1(\partial X)$ and $ N'_1$ in $Symp(\partial X, ker(\alpha_1))$. 
 Now we identify $E''-W' $ with $X  \cup T $ along the graph $\Gamma $ using Lemma~\ref{glueeasy}. We denote the resultant 4-manifold by $X' $.


\begin{figure}
\centering
\includegraphics[scale=.75]{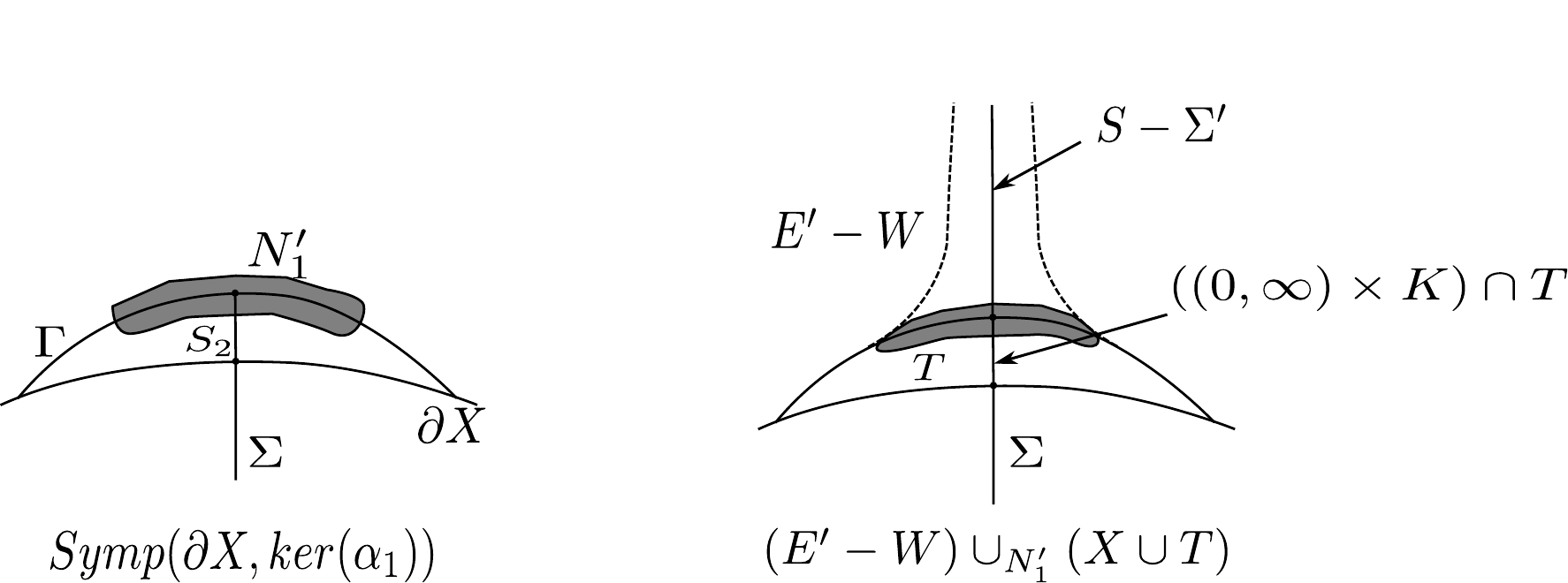}
\caption{Relative Symplectic Gluing}
\end{figure}

\noindent \textbf{Framings on the transverse knot:} \\
We will fix a contactomorphism $\phi_0$ from $(S^1 \times \mathbb{D}_{\delta}, Ker(d\theta + r^2 d\phi))$ to a neighborhood $S_2$ of transverse knot in $\partial W' $. We regard $\phi_0 $ as a base framing on the transverse knot $K$. We can see the effect of change of framing on knot as follows. 

 For $k \in \mathbb{Z}$, let $F_k: S^1 \times \mathbb{D}_{\delta} \rightarrow S^1 \times \mathbb{D}_s$ be a change of framing given by 
 $$(\theta, r, \phi) \mapsto (\theta, f(r), \phi+k \theta) $$
 where $f $ is a smooth function taking positive values. Let $F_k$ be a contactomorphism. Then we have, by definition of contactomorphism,
 $$ Ker (d\theta +r^2 d \phi) = \ Ker (d \theta + f^2(r)d(\phi + k \theta))  $$
A further simplification of the above shows that 
\begin{equation}
f(r) = \sqrt{\frac{r^2}{1-kr^2}}
\end{equation} 

Since, $(S^1 \times \mathbb{D}_{\delta}, Ker(d\theta, r^2d \phi) )$ is contactomorphic to $S_2$ which is a subset of $A \times \mathbb{D}_{\epsilon} $, where $\epsilon$ was chosen to get a symplectic dilation $X$. A tubular neighborhood of $ S_2$ is contained in $A \times \mathbb{D}_{\epsilon} $. Therefore, volume consideration implies that $\delta\leq C (\epsilon) $, for some positive number $C$. In other words, the size of a neighborhood of the transverse knot $K$ is bounded above. We can not have the symplectic dilation $X$ pointing into $E'' - W' $ along $S_2 $ and arbitrarily large volume of contact neighborhood $S_2 $ simultaneously.

From the equation (3.1), we see that the function f exists if $k < 1/C^2(\epsilon) $. 
Also, it  shows that for $k \geq 1/C^2(\epsilon) $, the function $f$ does not exist. This means that framings greater than $1/C^2(\epsilon) $ have too large a volume to be contained inside $A \times \mathbb{D}_{\epsilon}$.

\begin{remark}
On the other hand, negative values of $k$ decrease the volume of corresponding neighborhoods of $K$. In particular, we can carry out our construction of cap for all framings $k < 1/C^2(\epsilon) $.
\end{remark} 

\noindent \textbf{Extending Symplectic Surface $\Sigma $:}\\   
We need to extend the surface $\Sigma $ across $T$. We may assume that the symplectic dilation $v_1$ is tangent to $\Sigma $ in a neighbourhood of the knot $K$. This can be achieved by applying a perturbation to the symplectic surface $\Sigma $ near its boundary $K$, keeping the knot $K$ fixed pointwise.
    
 We notice that $(0, \infty) \times K$ is a symplectic surface in $Symp(\partial X, ker(\alpha_1)) $ since $K$ is a transverse knot in $\partial X $. Let $\Sigma_1 =T \cap ((0, \infty) \times K) $. Observe that the symplectic dilation $\frac{\partial}{\partial t }$ is tangent to the surface $\Sigma_1 $. Therefore, the surface $\Sigma $ in $X$ is extended to a smooth symplectic surface $\Sigma' = \Sigma \cup \Sigma_1 $. 
 
 We also assume that symplectic surface $S-int(\Sigma')$ in the relative symplectic is tangent to the symplectic dilation $v_2 $ near the knot $K$ by perturbing the $S-int(\Sigma')$ near the knot $K$. Notice that gluing the relative symplectic cap to $X$ extends $\Sigma \cup \Sigma_1$ to the (smooth) symplectic surface $S$ in $X'$ since symplectic dilations are matched near the knot $K$.  
\\
\textbf{Convexity Of The Boundary of $X'$ :}\\
 We notice that the $X'$ has convex boundary. Namely, the boundary of $E''-W' $ and $\partial X $ are convex hypersurfaces in $X'$. Moreover, contact forms $\alpha_1 $ and $\alpha_2 $ agree on intersection of these convex hypersurfaces $N'_1 $. Thus, we see that there is a symplectic dilation $v$ on the boundary of resultant manifold $(E''-W')\cup_{N'_1} (X \cup T) $ given by $v_1 $ (radial vector field) on the boundary of $E''-W' $ and by $ v_2$ on the boundary of $X \cup T$. In the neighborhood $N'_1$, where convex hypersurfaces intersect, an easy check shows that $v_1 $ agrees with $v_2$.  

Now, we use Symplectic Capping Theorem of Eliashberg and Etnyre (see Theorem \ref{STC} ) to get a closed 4-manifold. Notice that $(X,\Sigma)$ is capped symplectically.
\end{proof}

\begin{remark}
We emphasize that the choice of framing on the transverse knot $K$ is determined implicitly when we chose the neighbourhood $S_1 $ in $X$ \textit{contactomorphic} to $(\mathbb{S}^1 \times \mathbb{D}_{\delta}, Ker( d\theta + r^2 d \phi))$.
Therefore, the cut-and-paste operation of symplectically capping the pair $(X, \Sigma) $ can not be done with arbitrary choice of framing on the transverse knot $K= \partial \Sigma$. 
\end{remark}

\section{Minimum genus from symplectic caps}
We now prove Corollary~\ref{S3Cap} and Theorem~\ref{relThom}.

\begin{proof}[Proof of Corollary 1.1]: Notice that $\mathbb{B}^4 $ with symplectic structure given by $\omega_{st} $ has convex boundary $\partial \mathbb{B}^4 = \mathbb{S}^3 $. The radial vector field pointing out of $ \mathbb{B}^4$ is a symplectic dilation. By hypothesis, $K $ is a transverse knot in $(\mathbb{S}^3, \xi_{st}) $ and $S$ is a symplectic surface
in $\mathbb{B}^4 $ such that $S \cap \mathbb{S}^3 = K $. By Theorem \ref{cap}, there is pair $(Y, S') $, where $Y$ is a closed symplectic 4-manifold and $S'$ is a closed symplectic surface in it, such that the pair $(\mathbb{B}^4, S) $ can be capped symplectically to a pair $(Y, S') $. 
\end{proof}

\begin{proof}[Proof of Theorem~\ref{relThom}] Once we cap $(X, \Sigma) $ to a pair $(Y,S)$ symplectically. We apply symplectic Thom conjecture (see Theorem \ref{STC} ) to conclude that $\Sigma $ is genus minimizing in its relative homology class.
\end{proof}

\section{Existence of Symplectic Surfaces Bound By A Fibered Knot}

In this section, we start with a construction of a suitable Stein surface in which $K$ bounds a holomorphic curve. We perform Legendrian surgery on the boundary of the Stein surface to get a Stein filling of $(\mathbb{S}^3, \xi_{st}) $. Then, we show that $K$ bounds a symplectic surface in $(\mathbb{B}^4, \omega_{st}) $. We now turn to more details in the following proof. 

\begin{proof}[Proof of Theorem \ref{RelGiroux}]
Let $J$ denote an almost complex structure on a connected, orientable $\Sigma$ with nonempty boundary.  In dimension 2, every almost complex structure is integrable. Consider $\Sigma \times \mathbb{D}^2$ with product complex structure. We ``round the corners" to obtain contact boundary by following procedure.
We construct an exhausting plurisubharmonic function on $\Sigma \times \mathbb{D}^2$ as follows.

\begin{lemma} There is a Stein manifold $W_0 \subset \Sigma \times \mathbb{D}^2 $ with $(\partial W_0, \xi) $, where $\xi $ is the induced contact structure on $\partial W_0 $ as boundary of Stein manifold, is suppoterd by the open book $(\Sigma, id) $, where $id$ denotes the monodromy given by the identity map.
\end{lemma}

\begin{proof} We fix a symplectic form $\omega $ compatible with the almost complex structure $J$ induced by the product complex structure on $\Sigma \times \mathbb{D} $. Let $\psi$ be a subharmonic function given by Lemma \ref{pshsurf}. We now notice that for $z=(z_1,z_2) \in \Sigma \times \mathbb{D}^2 $ the function $\Phi(z)=\psi(z_1)+ |z_2|^2 $ is plurisubharmonic, as we can see by considering a chart $U \times V $.
The sublevel set given by $W_0= \{z | \Phi(z) \leq 1 \} $ is then a Stein domain.

To show that the induced contact structure on $\partial W_0 $ is supported by the open book
$(\Sigma, id)$, we recall following lemma 

\begin{lemma} The following statements are equivalent.
\begin{enumerate}
\item The contact manifold $(M, \xi) $ is supported by the open book $(B, \pi) $.
\item $(B,\pi)$ is an open book for $M$ and $\xi $ can be isotoped to be arbitrarily close (as oriented plane fields), on compact subsets of the pages, to the tangent planes to the the pages of the open book in such a way that after some point in the isotopy the contact planes are transverse to $B$ and transverse to the pages of the open book in a fixed neighborhood of $B$.
\end{enumerate} 

\end{lemma}

\begin{remark} \label{Cond2} The condition in part (2) of the above lemma involving transversality to the pages can be dispensed with for tight contact structures.
\end{remark}

Now we show that condition (2) of above lemma is satisfied by the open book $(\Sigma, id) $ and the induced contact structure $\xi$ on $\partial W_0 $.

First observe that the induced contact structure $\xi $ is Stein fillable and therefore it is tight. By remark \ref{Cond2}, it is sufficient to give an isotopy of plane fields which takes $\xi$ arbitrarily close, on compact subsets of the pages, to the tangent planes of the pages of the open book $(\Sigma, id)$. 
 
 By the construction of $W_0$, we see that contact structure induced on $\partial W_0$ as the boundary of $W_0$ is given by complex tangency in the tangent plane at every point. Let $(r, \theta) $ denote the polar coordinates on $\mathbb{D}^2 -\{0 \} $.
  Clearly, the radial vector field $\frac{\partial}{\partial r} $ in the disk is orthogonal to $\partial W_0 $ away from the knot $K=\partial \Sigma \times \{ 0\}$. Notice that $\frac{\partial}{\partial \theta}$ is tangent to $\partial W_0 $ and transverse to $\xi_p $. Also the vector field $\frac{\partial}{\partial \theta} $ is orthogonal to tangent planes of each page $\Sigma_{\theta} $. Therefore, we see that $\xi $ can be isotoped arbitrarily close to $T_p \Sigma_{\theta} $ by rotating normal to $\xi_p $ very close to $\frac{\partial}{\partial \theta}|_p $ on compact subsets of the pages. Hence $(\Sigma, id)$ supports the contact structure $\xi $. 
\end{proof}

 Also observe that $K$ bounds a holomorphic curve in $W_0$ given by $\Sigma \times \lbrace0\rbrace$ with $\partial \Sigma =K$. In particular, $K$ bounds a symplectic surface in $W_0 $. Before we attach 2-handles to $W_0 $, we recall following lemma to see the change in the open book due to symplectic handle addition.

\begin{lemma} [see 6.9 in \cite{Et3}]: Let $\gamma$ be a Legendrian knot on a page of the open book $(\Sigma, \phi)$ the contact manifold $(M, \xi)$. If $(M', \xi')$ is obtained by a Legendrian surgery along $\gamma $ then $(M', \xi') $ is supported by $(\Sigma, \phi \circ D_{\gamma}) $.
\end{lemma}

 By LeRP, $ \gamma_1$ is isotopic to a Legendrian curve with contact framing, which we denote by $\gamma_1 $ itself.  Performing Legendrian surgery along $\gamma_1 $ we get the symplectic manifold $W_1 $ with the boundary $\partial W_1 $ as contact manifold supported by the open book $(\Sigma, D_{\gamma_1}) $. Inductively, we perform surgery along curve $\gamma_i $ to get the symplectic manifold $W_i $ with the contact boundary $\partial W_i $ supported by the open book $(\Sigma, D_{\gamma_1} \circ..\circ D_{\gamma_i}) $. Notice that $ W_{i-1} $ embeds symplectically into $W_i $. Thus, $ \Sigma \times \lbrace 0 \rbrace$ is symplectically embedded into $W_i $. Finally we get $W_k $ with contact boundary $\partial W_k $ supported by  $(\Sigma, D_{\gamma_1} \circ D_{\gamma_2}\circ..\circ D_{\gamma_k}) $. Thus we obtain $\mathbb{S}^3 $ as a boundary of $W_k $ and by hypothesis $\xi_{st} $ is supported by $(\Sigma, D_{\gamma_1} \circ D_{\gamma_2}\circ..\circ D_{\gamma_k}) $. Thus, we obtain a symplectic filling of $(\mathbb{S}^3, \xi_{st}) $.

\begin{remark}
Notice that we performed a Legendrian surgery on a Stein surface at each stage, thus result is a Stein surface. The manifolds $W_i $'s are Stein surfaces. There is a Stein embedding of $W_i $ into $W_{i+1} $, for $i =0,1,.., k-1 $.
\end{remark}
 By uniqueness of Stein filling of $(\mathbb{S}^3, \xi_{st}) $ due to Eliashberg, we see that the symplectic filling $W_k $ is equivalent to $\mathbb{B}^4 \subset \mathbb{C}^2$ as a Stein surface. In particular, $W_k $ is symplectomorphic to $(\mathbb{B}^4, \omega_{st}) $.
 Therefore, $K$ bounds a symplectic surface in $(\mathbb{B}^4,\omega_{st})$. 
We denote the image of $\Sigma$ under this symplectomorphism by $\Sigma' $.

\begin{remark} As we have seen earlier, $K$ bounds a holomorphic curve in $W_k $, but it may not bound a holomorphic curve in $\mathbb{B}^4 $. Under an equivalence of Stein surfaces, the image of holomorphic curve is not necessarily a holomorphic curve.
\end{remark} 
  
\end{proof}

\begin{proof}[Proof of Corollary \ref{4G3G}]
By construction of symplectic filling of $(\mathbb{S}^3, \xi) $ in Theorem \ref{RelGiroux}, we notice that the genus of symplectic surface $\Sigma' $ bound the knot $K$ equals the genus of the page of the open book $(\Sigma, \phi) $. Combining Theorem \ref{cap} and \ref{relThom} we see that symplectic surface $\Sigma' $ in $\mathbb{B}^4 $ minimizes genus in its relative homology class. Thus, $g_4(K)=g(\Sigma')$ which further implies $g_4(K)=g(K) $.     
\end{proof}

\section{A Criterion For Strong Quasipositivity }

In this section, we recall briefly notions of quasipositive knots and strongly quasipositive knots.
Then, we give a criterion for fibered knot to be strongly quasipositive.

\begin{definition} A knot is called \textit{quasipositive} if it can be realized as closure of a braid given by product of conjugates of generators. 
\end{definition}

\begin{definition} A knot is called \textit{strongly quasipositive} if it has quasipositive Seifert surface. And \textit{quasipositive Seifert surface} is a surface obtained by taking parallel disks and attaching positive bands to them. 
\end{definition}

For more details we refer to Rudolf's survey article (see \cite{R1}). We also refer to Hedden's article on notions of quasipositivity (see \cite{H1}).  In his article, Hedden characterises fibered knots $K$ in $\mathbb{S}^3 $ for which $ g(K)= g_4(K)$ (see Corollary 1.6 in \cite{H1}). 

\begin{theorem}[Livingston \cite{L1}]If $K$ is stronlgy quasipositive then $g(K)= g_4(K)=\tau(K) $, where $\tau(K) $ is concordance invariant.
\end{theorem}

Hedden proves the converse of the above theorem when $K$ is a fibered knot.

\begin{theorem}[Hedden \cite{H1}]If $K$ is fibered knot in $\mathbb{S}^3 $ then $g(K)= g_4(K)= \tau(K)$ if and only if $K$ is strongly quasipositive.
\end{theorem}

\begin{proof}[proof of Corollary \ref{sqp}] Now we combine the Theorem \ref{RelGiroux} and Corollary \ref{4G3G} to conclude that if we have a fibered knot with monodromy given by product of positive Dehn twists then $g(K)= g_4(K) $. Now applying a corollary to Hedden's theorem (See 1.6 in \cite{H1}) we see that fibered knot $K$ must be strongly quasipositive. \end{proof}

\begin{remark} We notice that the hypothesis for above criterion for strong quasipositivity is in terms of monodromy of the fibration $\mathbb{S}^3-K$ over $\mathbb{S}^1 $. Since monodromy is product of positive Dehn twists the corresponding open book decomposition $(\mathbb{S}^3, K) $ is Stein fillable and therefore the open book $(\Sigma, \phi) $ must support tight contact structure. By uniqueness of tight contact structure (up to isotopy) on $\mathbb{S}^3$, it must be the standard tight contact structure $\xi_{st} $ on $\mathbb{S}^3 $. And fibered knot $K$ supports the standard tight contact structure $ \xi_{st}$ if and only if $K$ is strongly quasipositive (See proposition 2.1 in \cite{H1} ). So we may have concluded that $K$ is strongly quasipositive without using theorem 1.1 and corollary 1.1 .    
\end{remark}

\begin{remark}In his survey article, Rudolf remarks (See last paragraph in section 3.3.7 in \cite{R1} ) that little is known about systematic construction of nonstrongly quasipositive knots. In the light of above remark, it is useful to know that a fibered knot with monodromy given by a product of positive Dehn twists must be strongly quasipositive.  
\end{remark}

\bibliographystyle{amsplain}
\bibliography{ref}
\end{document}